\documentclass[]{amsart}
\usepackage{amssymb}
\usepackage{amsxtra}
\usepackage{amsmath}
\usepackage{amsfonts}

\numberwithin{equation}{section}

\newtheorem{thm}{Theorem}[section]

\newtheorem{prop}[thm]{Proposition}

\newtheorem{lem}[thm]{Lemma}

\newcommand{\eqa}{\begin{eqnarray}}
\newcommand{\eeqa}{\end{eqnarray}}
\newcommand{\beq}{\begin{equation}}
\newcommand{\eeq}{\end{equation}}
\newcommand{\nn}{\nonumber}
\newcommand{\p}{\partial}

\def \la {\langle}
\def \ra{\rangle}
 \def \dsum{\displaystyle\sum}

\begin{document}
\title[]{ A family of conformally flat Hamiltonian-minimal Lagrangian tori in $\mathbb{CP}^3$ }
 \author[]{Andrey E. Mironov and Dafeng Zuo }

\address{Mironov, Sobolev Mathematical institute,
 Siberian Branch of the Russian Academy of Sciences and
 Novosibirsk state university, Novosibirsk}

\address{Zuo, School of Mathematics, Korea Institute for Advanced Study
207-43 Cheongnyangni 2-dong, Dongdaemun-gu Seoul, 130-722 Korea
 and Department of Mathematics, University of Science and
Technology, Hefei 230026, P.R.China}

\email{mironov@math.nsc.ru, dfzuo@kias.re.kr}

\date{\today}

\maketitle
\begin{abstract}In this paper by reduction we construct a family of conformally flat
Hamiltonian-minimal Lagrangian tori in $\mathbb{CP}^3$ as the image
of the composition of the Hopf map
$\mathcal{H}: \mathbb{S}^7\to \mathbb{CP}^3$
and a map $\psi:\mathbb{R}^3 \to \mathbb{S}^7$ with certain conditions.\\~
\\~
{\bf MSC(2000):} 05C42, 53D12, 35Q51;\\
~
 {\bf Key words:} {conformally flat, Hamiltonian-minimal Lagrangian tori}
\end{abstract}
%%%%%%%%%%%%%%%%%%%%%%%%%%%%%%%%%%%%%%%%%%%%%%%%%%%%%%%%%%%%%%%%%%%%%%%%%%%%%%%%%%%5
\section{Introduction}
A Lagrangian submanifold of a K\"ahler manifold is said to be
Hamiltonian-minimal (briefly, H-minimal)  if it is a critical
point of the volume under Hamiltonian deformations. This notion
was introduced by Oh in \cite{Oh}, who also gave an example, that
is, Clifford tori in $\mathbb{C}^n$ with standard Hermitian metric
%\beq  \mathbb{S}^1(r_1)\times\dots\times \mathbb{S}^1(r_n)\subset{\mathbb C}^n,\eeq
$$
 \mathbb{S}^1(r_1)\times\dots\times
 \mathbb{S}^1(r_n)\subset{\mathbb C}^n,
$$
 where $\mathbb{S}^1(r_j)$
is a circle of radius $r_j$ in ${\mathbb C}$.

In \cite{HR1}, H\'elein and Romon studied a general construction of H-minimal tori
in $\mathbb{C}^2$ from the point of view of completely integrable systems.
They provided new explicit nontrivial examples of H-minimal
Lagrangian tori which include the examples previously constructed by Castro and Urbano
 in $\Bbb C^2$ \cite{CU2}. A similar construction has been generalized to the cases
 of Hermitian symmetric spaces, e.g., in $\mathbb{CP}^2$, see \cite{HR,HM1,HM2} for details.
 However, in the non-flat cases, the underlying equations are no longer linear, which makes
 the problem much harder.  Although in \cite{HM2}, Ma introduced a spectral
 parameter $\lambda \in \mathbb{S}^1$, as she pointed out, a description of H-minimal
 Lagrangian tori in
  $\mathbb{CP}^2$ in terms of theta functions seems to be possible. But owing to this spectral
 parameter $\lambda \not\in  \mathbb{C}$,  thus it is still open about the integrability of
 this problem in classical sense. In \cite{M2,M3} it is shown that if a Lagrangian
conformal map from ${\mathbb R}^2$ to ${\mathbb C}P^2$ is given as
composition of maps
$\varphi:={\mathcal H}\circ \psi:{\mathbb R}^2\rightarrow S^5\rightarrow{\mathbb
C}P^2,$
where ${\mathcal H}$ is the Hopf map, then components of $\varphi$
satisfy Shr\"odinger equation
$$\Delta\varphi_j+i(\theta_x\partial_x\varphi_j+\theta_y\partial_y\varphi_j)+4e^u\varphi_j=0,$$
where $ds^2=2e^u(dx^2+dy^2)$ is an induced metric and $\theta(x,y)$ is
Lagrangian angle. In the case of H-minimal Lagrangian tori $\theta$ is a linear function.
So in order to construct finite gap tori it is necessary to use spectral data of finite gap Shr\"odiger
operators, see \cite{DKN}, \cite{VN} for details.

In \cite {CU}, Castro and Urbano constructed  a family of minimal Lagrangian tori
in $\mathbb{CP}^2$, which are characterized by their invariance under a one-parameter
 group of holomorphic isometries of $\mathbb{CP}^2$. By using this idea, in \cite{M2,HM3},
 they independently reduced this problem to a  one dimensional system and
 obtained an equivariant solution in terms of elliptic functions, and then
 constructed  H-minimal Lagrangian tori in $\mathbb{CP}^2$.

In high dimensional case one of the present authors constructed some
examples of  H-minimal and minimal Lagrangian immersions and
embeddings in ${\mathbb C}^n$ and $\mathbb{CP}^n$. Recently we can
find some of works about this topic, see \cite{CIU, CG,DO} and
references therein. But it is far from the complete
characterization.

 In this paper we address to construct a family of conformally flat
 H-minimal Lagrangian tori in $\mathbb{CP}^3$ by reduction methods,
 which generalizes the results in \cite{M2,M4,HM3},  with the metric
$$
 ds^2=e^u(dx^2+dy^2+dz^2).
$$
In $\mathbb{CP}^3$, it seems to be a little harder. We thus
restrict to discuss a very special case, that is, $u=u(z)$ and the
Lagrangian angle $\theta=ax+by$, where $a$ and $b$ are arbitrary
real constants.

%%%%%%%%%%%%%%%%%%%%%%%%%%%%%%%%%%%%%%%%%%%%%%%%%%%%%%%%%%%%%%%%%%%%%%%%%%%%%%%%%%%%%%%%%%%
\section{Preliminaries}

In this section, we review some well-known facts without proofs and
sketch our strategy about the construction of H-minimal Lagrangian
cone or tori in $\mathbb{C}^4$ and $\mathbb{CP}^3$.

\subsection{Notations} Let $\mathbb{C}^4$ be the canonical complex space of dimension $4$ endowed
with an Hermitian product $ \langle u,v
\rangle=\dsum_{k=1}^4u_k\bar{v}_k.$ Let us denote
$\omega=\mbox{Im} \langle ~,~\rangle$ and $(~,~)=\mbox{Re} \langle
~,~\rangle$. Let $\psi: \mathbb{R}^3\to \mathbb{S}^7$ be an
oriented Lagrangian immersion $L$, i.e. $\psi^* \omega=0$, where
$\mathbb{S}^7$ is the unit sphere in ${\mathbb C}^4$. Wolfson in
\cite{Wol} introduced a Lagrangian angle $\theta$ and obtained a
criterion of H-minimality of $L$ in terms of $\theta$, that is,
{\it the Lagrangian immersion $L$ is H-minimal if and only if the
Lagrangian angle $\theta$ is a harmonic function on $L$.} The
Lagrangian angle $\theta$ of $L$ in $\mathbb{C}^4$ is defined by
the formula
$$
 e^{i\,\,\theta(p)}=dz_1\wedge\dots \wedge
 dz_4(\Psi), \quad p\in L, %\label{MZ1}
$$
 where $z_j,j=1,\cdots,
4$ are coordinates on ${\mathbb C}^4$ and
 $\Psi$ is an orthonormal tangent frame at $p\in L$ with the same
 orientation of $L$. For the general case, see \cite{Wol} for details.

Let $\mathcal{H}:S^7\rightarrow {\mathbb C}P^3$ be the Hopf map.
An induced Hermitian product on ${\mathbb CP}^3$ is  called to be
the Fubini-Study metric defined by
$<\zeta_1,\zeta_2>:=\la \tilde{\zeta}_1,\tilde{\zeta}_2\ra,$
 where $\zeta_i,~i=1,2$ are tangent to ${\mathbb CP}^3$ and
 $\tilde{\zeta}_i$ are the corresponding horizontal lifting by $\mathcal{H}$.
Let $\mathcal{C}$ be a Lagrangian cone in ${\mathbb C}^4$ with the vertex at the origin.
It follows from the definition of $<~,~>$ that $\mathcal{H}(\widetilde{\mathcal{C}})$ is a Lagrangian
submanifold in ${\mathbb CP}^3$, where $\widetilde{\mathcal{C}}=\mathcal{C} \cap S^7$.
Moreover, if the cone $\mathcal{C}$ is H-minimal in $\mathbb{C}^4$, then
$\mathcal{H}(\widetilde{\mathcal{C}})$ is also H-minimal in ${\mathbb
CP}^3$, see \cite{M1,CIU} for details.

\subsection{On conformally flat Lagrangian immersions}
In the following we only consider conformally flat
immersions in $\mathbb{C}^4$ and $\mathbb{CP}^3$.

Let $\psi=(\psi^1,\psi^2,\psi^3,\psi^4): \mathbb{R}^3\to \mathbb{S}^7\subset \mathbb{C}^4$
 be an oriented immersion with a
conformally flat metric
$$
 ds^2=e^{u(x,y,z)}(dx^2+dy^2+dz^2)
$$
satisfying the following properties
$$
  \la \psi,\psi_x\ra=\la \psi,\psi_y\ra=\la \psi,\psi_z\ra=0,
$$
$$
 \la \psi_x,\psi_y\ra= \la \psi_y,\psi_z\ra=\la
 \psi_z,\psi_x\ra=0,
$$
 by the above arguments, thus
$\mathcal{H}\circ\psi$ is a Lagrangian immersion in
$\mathbb{CP}^3$ and \beq
\Phi=(\psi,e^{-u}\psi_x,e^{-u}\psi_y,e^{-u}\psi_z)^{t}\in
\mbox{U}(4)\label{MZ7}\eeq and the Lagrangian angle
$\theta(x,y,z)$ is given by \beq e^{i\,\,\theta}={\rm
det}(\Phi).\label{MZ8}\eeq By using \eqref{MZ7} and \eqref{MZ8},
we have \beq
\Psi=(e^{i\,\,\theta}\psi,e^{-u}\psi_x,e^{-u}\psi_y,e^{-u}\psi_z)^{t}\in
\mbox{SU}(4)\label{MZ9}\eeq and denote \beq
\mathcal{U}:=\Psi_x\Psi^{-1},\quad  \mathcal{V}:=\Psi_y\Psi^{-1},
\quad \mathcal{W}:=\Psi_z\Psi^{-1} \in
\mbox{SU}(4).\label{MZ10}\eeq The compatibility condition of
\eqref{MZ10} is \beq
\mathcal{U}_y-\mathcal{V}_x+[\mathcal{U},\mathcal{V}]=0, \,
     \mathcal{V}_z-\mathcal{W}_y+[\mathcal{V},\mathcal{W}]=0,\,
     \mathcal{W}_x-\mathcal{U}_z+[\mathcal{W},\mathcal{U}]=0.
\label{MZ11}\eeq Moreover, if $\triangle \theta(x,y,z)=0$, then
the immersion $\mathcal{H}\circ\psi$ is a conformally flat
H-minimal Lagrangian immersion in $\mathbb{CP}^3$,  where
$\triangle$ is the corresponding Lapalacian operator given by the
formula
$$
 \triangle:=-e^{-u(z)}\,[\,\frac{\p^2}{\p
 x^2}+\frac{\p^2}{\p y^2}+ \frac{\p^2}{\p
 z^2}+\frac{u_x}{2}\frac{\p}{\p x}+\frac{u_y}{2}\frac{\p}{\p y}
 +\frac{u_z}{2}\frac{\p}{\p z}\,].
$$

Conversely, given $\mbox{SU}(4)$-valued $\mathcal{U}, \ \mathcal{V},\ \mathcal{W}$
satisfying \eqref{MZ11}, we solve the system \eqref{MZ10} with \eqref{MZ9}
for $\Psi\in \mbox{SU}(4)$ and then obtain the immersion $\psi:\mathbb{R}^3\to \mathbb{S}^7$ and the
Lagrangian angle $\theta$. If $\triangle\theta=0$ and $\psi$ satisfies
certain periodic conditions, then the immersion
$\mathcal{H}\circ\psi$ gives a conformally flat H-minimal Lagrangian torus
in $\mathbb{CP}^3$.  Notice that for general case, this method does not work.

%%%%%%%%%%%%%%%%%%%%%%%%%%%%%%%%%%%%%%%%%%%%%%%%%%%%%%%%%%%%%%%%%%%%%%%%%%%%%%%%%%%%%%%%%%%
\section{Conformally flat H-minimal Lagrangian tori
in $\mathbb{CP}^3$}
%%%%%%%%%%%%%%%%%%%%%%%%%%%%%%%%%%%%%%%%%%%%%%%%%%%%%%%%%%%%%%%%%%%%%%%%%%%
In this section, by using the above method, we will construct a special class
of conformally flat H-minimal Lagrangian tori in $\mathbb{CP}^3$ with $u=u(z)$
and the Lagrangian angle $\theta=ax+by$, where $a$ and $b$ are arbitrary
real constants.

\subsection{}The first step is to choose $\mathcal{U}$, $\mathcal{V}$ and $\mathcal{W}$.
By a direct calculation, we have the following lemma.

\begin{lem} Suppose  $\mathcal{U}=(u_{kl})$,
 $\mathcal{V}=(v_{kl})$, $\mathcal{W}=(w_{kl})$ are $\hbox{SU}(4)$-valued matrices
satisfying \eqref{MZ10} and $u_{kl}, v_{kl},w_{kl}$ depend only one variable $z$ for
 $2\leq k,l\leq 4$, then they must be the following form
$$
  \small{\mathcal{U}=
  \left(
  \begin{array}{cccc}
   i a & e^{i\, \theta+u} & 0 & 0\\
   -e^{-i\, \theta+u} & -i(a+c_1) & ic_2 & ic_3e^{-3u}-u'\\
   0 & ic_2 & ic_1 & 0 \\
   0 & ic_3e^{-3u}+u' & 0 & 0 \\
  \end{array}\right),} %\label{MZ3.4}
$$
and
$$
 \small{\mathcal{V}=
  \left(
  \begin{array}{cccc}
  ib & 0 & e^{i\, \theta+u} & 0\\
  0 & ic_2 & ic_1 & 0\\
   -e^{-i\, \theta+u} & ic_1 & -i(b+c_2) & ic_3e^{-3u}-u' \\
   0 & 0 & ic_3e^{-3u}+u' & 0 \\
  \end{array}\right),}
$$
and
$$
 \small{\mathcal{W}=
  \left(
  \begin{array}{cccc}
   0 & 0 & 0 & e^{i\, \theta+u}\\
   0 & ic_3e^{-3u} & 0 & 0\\
   0 & 0 & ic_3e^{-3u} & 0 \\
   -e^{-i\, \theta+u} & 0 & 0 & -2ic_3e^{-3u} \\
  \end{array}\right)},
$$
  where
  $u=u(z)$ satisfies the equation
  \beq u'\,^2+e^{2u}+c_3^2e^{-6u}-\mathfrak{C}=0, ~~u'=\frac{du}{dz}.\label{MZ3.3}\eeq
and $c_1$, $c_2$, $c_3$ are arbitrary real constants and $\mathfrak{C}=ac_1+bc_2+2c_1^2+2c_2^2$.
\end{lem}

\subsection{}The second step is to solve the following system
\beq \Psi_x=\mathcal{U}\Psi, \quad \Psi_y=\mathcal{V}\Psi,\quad
\Psi_z=\mathcal{W}\Psi.\label{MZ3.7}\eeq
Write
\beq \Psi=(e^{i\,\,\theta}\psi,e^{-u}\psi_x,e^{-u}\psi_y,e^{-u}\psi_z)^{t},
\label{MZ3.8}\eeq
where $\psi=\psi(x,y,z)$ is a smooth function. By using \eqref{MZ3.8}, the system
\eqref{MZ3.7} can be rewritten as
\eqa
&&\psi_{xz}-(u'+i c_3 e^{-3u})\psi_x=0, \label{MZ3.9}\\
&&\psi_{yz}-(u'+i c_3 e^{-3u})\psi_y=0, \label{MZ3.10}\\
&&\psi_{xy}-i(c_2\psi_x+ic_1\psi_y)=0,\label{MZ3.11}
\eeqa
and
\eqa
&&\psi_{xx}+ e^{2u}\psi+i(a+c_1)\psi_x-ic_2\psi_y+((u'-i c_3 e^{-3u})\psi_z=0,\label{MZ3.13}\\
&&\psi_{yy}+ e^{2u}\psi-ic_1\psi_x+i(c_2+b)\psi_y+((u'-i c_3 e^{-3u})\psi_z=0,\label{MZ3.14}\\
&&\psi_{zz}+ e^{2u}\psi+(2ic_3e^{-3u}-u')\psi_z=0. \label{MZ3.12}
\eeqa

From \eqref{MZ3.9} and \eqref{MZ3.10}, we know that $\psi$ must be
of the form \beq \psi=P(z)\varphi(x,y)+Q(z), \quad \varphi(x,y)\ne
\mbox{constant}, \label{MZ3.15}\eeq and then \eqref{MZ3.9} and
\eqref{MZ3.10} reduce to \beq P(z)-(u'+i c_3
e^{-3u})P(z)=0.\label{MZ3.16}\eeq The solution of \eqref{MZ3.16} is
\beq P(z)=a_1e^{u+ic_3\int e^{-3u}dz},\quad a_1\in
\mathbb{R}.\label{MZ3.17}\eeq

Substituting \eqref{MZ3.15} and \eqref{MZ3.17} into
\eqref{MZ3.11}, we get \beq
\varphi_{xy}-i(c_2\varphi_x+c_1\varphi_y)=0.\label{MZ3.18}\eeq
Substituting \eqref{MZ3.15} into \eqref{MZ3.13} and
\eqref{MZ3.14}, and then using \eqref{MZ3.3} and \eqref{MZ3.18},
we get the following \eqa
&& Q'(z)=\frac{e^{5u}Q(z)}{ic_3-u'e^{3u}},\label{MZ3.20}\\
&& \varphi_{xx}+i(a+c_1)\varphi_x-ic_2\varphi_y+\mathfrak{C}\,\varphi=0,\label{MZ3.21}\\
&& \varphi_{yy}-ic_1\varphi_x+i(b+c_2)\varphi_y+\mathfrak{C}\,\varphi=0.\label{MZ3.22}
\eeqa
Write \beq Q(z)=H(z)e^{iG(z)},\label{MZ3.19}\eeq
where $G(z)$ and $H(z)$ are real smooth functions.
By differentiating \eqref{MZ3.3}, we obtain
\beq u''+e^{2u}-3c_3^2e^{-6u}=0.\label{MZ3.1}\eeq
By using \eqref{MZ3.1}, and separating the real part and the imaginary part of \eqref{MZ3.20}, we obtain
$$
 H'(z)(\mathfrak{C}-e^{2u})+u'e^{2u}H(z)=0,\quad
 G'(z)(e^{2u}-\mathfrak{C})-c_3e^{-u}=0,\label{MZ3.23}
$$
thus
$$
 H(z)=a_2\sqrt{\mathfrak{C}-e^{2u}}, ~~a_2\in
 \mathbb{R},\quad
 G(z)=\int \frac{c_3e^{-u}}{e^{2u}-\mathfrak{C}}dz+a_3,~~a_3\in \mathbb{R}.
$$

From \eqref{MZ3.18}, without loss of generality, we could assume
that $\varphi(x,y)$ has the form
$$
 \varphi(x,y)=\sum
 a_{\alpha_j\beta_j}e^{i(x\alpha_j+y\beta_j)},
 ~~\beta_j=\frac{c_2\alpha_j}{\alpha_j-c_1},~~
 a_{\alpha_j\beta_j}\in \mathbb{C}, ~~a_{00}=0.
$$
It follows from  \eqref{MZ3.19},\eqref{MZ3.21} and \eqref{MZ3.22}
that $\alpha=\alpha_j$ is a root of the equation
$$
{\alpha}^{3}+a{\alpha}^{2}- \mathfrak{B}
\alpha+{c_{{1}}}\mathfrak{C} =0.
$$
 where $\mathfrak{B}=
2\,c_{{1}}a+3\,{c_{{1}}}^{2}+c_{{2 }}b+3\,{c_{{2}}}^{2} $.

Notice that up to now we only use \eqref{MZ3.9}---\eqref{MZ3.14} to obtain an explicit form of
$\psi$ as follows
\eqa \psi(x,y,z)&&=\sum a_1a_{\alpha_j\beta_j}e^{u(z)+ic_3 \int{e^{-3u(z)}dz}}
 e^{i(x\alpha_j+y\beta_j)}\nn\\&&\quad + ~a_2e^{ia_3}\sqrt{\mathfrak{C}-e^{2u(z)}}e^{ic_3
 \int{\dfrac{e^{-u(z)}}{e^{2u(z)}-\mathfrak{C}}}dz}.\label{MZ3.27}\eeqa
 Furthermore, it is easy to check that this function $\psi$ also satisfies \eqref{MZ3.12}.
 Thus we solve the system \eqref{MZ3.7}. Summarizing the above discussions,
 we have the following proposition.

\begin{prop}If we suppose that $u=u(z)$ is a smooth solution of
$u(z)'\,^2+e^{2u(z)}+c_3^2e^{-6u(z)}-\mathfrak{C}=0$; and $\alpha$
is a root of the equation \beq {\alpha}^{3}+a\,{\alpha}^{2}-
\mathfrak{B}\,\alpha +{c_{{1}}}\mathfrak{C} =0\label{MZ3.30} \eeq
Then
$\Psi=(e^{i\,\,\theta}\psi,e^{-u}\psi_x,e^{-u}\psi_y,e^{-u}\psi_z)^{t}$
is a solution of the system \eqref{MZ3.7} with $\theta=a\,x+b\,y$
and \beq \psi(x,y,z)=\kappa_1 e^{i(\alpha\, x +\beta\, y)} P(z)+
\kappa_2 Q(z),\quad \beta=\dfrac{c_2\alpha}{\alpha-c_1},\nn\eeq
where $\kappa_1$ and $\kappa_2$ are arbitrary complex constants
and
$$
 P(z)=e^{u(z)+ic_3 \int{e^{-3u(z)}dz}},\quad
 Q(z)=\sqrt{\mathfrak{C}-e^{2u(z)}}e^{ic_3
 \int{\dfrac{e^{-u(z)}}{e^{2u(z)}-\mathfrak{C}}}dz}.
$$

\end{prop}
\subsection{Main results} We now state our main theorem.
\begin{thm}\label{thm3.3}
Suppose that the equation \eqref{MZ3.30} has three distinct roots,
denoted by $\alpha_1$, $\alpha_2$ and $\alpha_3$.
Write
 $\beta_j=\frac{c_2\alpha_j}{\alpha_j-c_1},~j=1,2,3$.
Then the  map $\mathcal{H}\circ \psi: \mathcal{R}^3\to \mathbb{CP}^3$
defines a conformally flat H-minimal Lagrangian immersion in $\mathbb{CP}^3$,
where $\mathcal{H}: \mathbb{S}^7\to \mathbb{CP}^3$ is the Hopf map and
the map $\psi: \mathcal{R}^3\to \mathbb{S}^7\subset \mathbb{C}^4$ is given by  the formula
\beq \psi=(\gamma_1 P(z)e^{i(x\alpha_1+y\beta_1)},\  \gamma_2 P(z)e^{i(x\alpha_2+y\beta_2)},\
 \gamma_3 P(z)e^{i(x\alpha_3+y\beta_3)}, \ \gamma_4 Q(z)). \nn\eeq
 Here $\gamma_4=\sqrt{\frac{1}{\mathfrak{C}}}$ and

  $\tiny{\gamma_1=\sqrt{\frac{\mathfrak{C}+\alpha_2\alpha_3}{\mathfrak{C}(\alpha_1-\alpha_2)(\alpha_1-\alpha_3)}}\,,
\,\gamma_2=\sqrt{\frac{\mathfrak{C}+\alpha_1\alpha_3}{\mathfrak{C}(\alpha_2-\alpha_1)(\alpha_2-\alpha_3)}}\,,
\, \gamma_3=\sqrt{\frac{\mathfrak{C}+\alpha_1\alpha_2}{\mathfrak{C}(\alpha_3-\alpha_1)(\alpha_3-\alpha_2)}}}\,.$

\end{thm}

\begin{proof} It suffices to check that
\beq \Psi=(e^{i\,\,\theta}\psi,e^{-u}\psi_x,e^{-u}\psi_y,e^{-u}\psi_z)^{t} \in \mbox{SU}(4).\nn\eeq
By using \eqref{MZ3.30}, we have
\beq \alpha_1+\alpha_2+\alpha_3=-a,~~ \alpha_1\alpha_2+
\alpha_1\alpha_3+\alpha_2\alpha_3=-\mathfrak{B},~~
\alpha_1\alpha_2\alpha_3=-\mathfrak{C}. \label{Pf1} \eeq
It follows from \eqref{Pf1} and the explicit forms of $\gamma_j$ that
\eqa & \dsum_{j=1}^3\gamma_j^2=\gamma_4^2, &\quad  \sum_{j=1}^3\gamma_j^2\alpha_j=0,\quad
 \dsum_{j=1}^3\gamma_j^2\alpha_j^2=1 ,\nn \\
& \dsum_{j=1}^3\gamma_j^2\alpha_j\beta_j=0, &
 \quad \dsum_{j=1}^3\gamma_j^2\beta_j=0, \quad\sum_{j=1}^3\gamma_j^2\beta_j^2=1.\nn
\eeqa
These identities yield that
\eqa
&& \la \psi,\psi\ra=1, \la \psi_x,\psi_x\ra=\la \psi_y,\psi_y\ra=e^{2u},\nn\\
&& \la \psi,\psi_x\ra=\la \psi,\psi_y\ra=\la \psi,\psi_z\ra=0,\nn\\
&& \la \psi_x,\psi_y\ra= \la \psi_y,\psi_z\ra=\la \psi_z,\psi_x\ra=0, \nn\\
&& \la \psi_z,\psi_z\ra = P'(z)\overline{P'(z)}\sum_{j=1}^3\gamma_j^2 +
      \gamma_4^2 Q'(z)\overline{Q'(z)}  \nn\\
&&\qquad  =\gamma_4^2 [e^{2u}(u'2+c_3^2e^{-6u})+\frac{e^{4u}(u'2+c_3^2e^{-6u})}{\mathfrak{C}-e^{2u}}]\nn\\
&&\qquad =e^{2u}. \qquad \mbox{by using \eqref{MZ3.3}}  \nn\eeqa
That is to say, $\Psi \in \mbox{SU}(4)$. Thus we complete the proof of the theorem.\end{proof}

We finish this section to discuss how to obtain conformally flat H-minimal
Lagrangian tori in $\mathbb{CP}^3$.

Notice that in \eqref{MZ3.3} if we make the following change
 \beq u=u(z):=-\log(2\sqrt{-q(z)}\,),\label{MZ3.2}\eeq
then we have
\beq q'(z)^2=256c_3^2q(z)^5+4\mathfrak{C}q(z)^2+q(z).\label{MZ4.1}\eeq
Thus if we choose three real constants $c_1$, $c_2$ and $c_3$
such that the equation $$256c_3^2t^5+4\mathfrak{C}t^2+t=0$$
has two negative roots and does not have multiple roots,
then this assures that \eqref{MZ4.1} has a smooth periodic solution of the period $\tau$,
see \cite{NO} for details. It follows from \eqref{MZ3.2} that so is \eqref{MZ3.3}.
We here remark that in this case $\mathfrak{C}=ac_1+bc_2+2c_1^2+2c_2^2>0$.

We next discuss the condition such that
 the function $\psi$ in \eqref{MZ3.27} is a periodic function of $x$, $y$,$z$ repesctively.
According to the form of $\psi$ in \eqref{MZ3.27}, if  we assume
that $c_1\in \mathbb{Q}$ and $\alpha_2$, $\alpha_2$ and $\alpha_3$
are  three distinct rational roots of \eqref{MZ3.30},  then $\psi$
is periodic w.r.t. $x$ and $y$. Notice that $u(z+\tau)=u(z)$ and
there exists a periodic function $h(z)$ of the periodic $\tau$
such that
$$
 \int
 \dfrac{e^{-3u(z)}\mathfrak{C}}{e^{2u(z)}-\mathfrak{C}}dz=h(z)+z
 \int_0^\tau
 \dfrac{e^{-3u(z)}\mathfrak{C}}{e^{2u(z)}-\mathfrak{C}}dz.
$$
This implies that if we assume
$$
 \frac{c_3\mathfrak{C}\tau}{2\pi}
 \int_0^\tau \dfrac{e^{-3u(z)}}{e^{2u(z)}- \mathfrak{C}}dz \in
 \mathbb{Q}\,,
$$
then $\psi$ is periodic in $z$ with the period $n \tau$ for some
$n\in \mathbb{N}$.

Thus, combining with Theorem \ref{thm3.3}, we have
\begin{thm}If we suppose that

$1$. $u=u(z)$ is a periodic solution of \eqref{MZ3.3}
with the period $\tau$;

$2$. $c_1\in \mathbb{Q},\quad \dfrac{c_3\mathfrak{C}\tau}{2\pi}
\int_0^\tau \dfrac{e^{-3u(z)}}{e^{2u(z)}-
\mathfrak{C}}dz \in \mathbb{Q}\,$;

$3$. $\alpha_1$,\ $\alpha_2$ and $\alpha_3$ are distinct rational roots of \eqref{MZ3.30}.

 Then the map $\mathcal{H}\circ \psi: \mathbb{R}^3\to \mathbb{CP}^3$
defines a conformally flat H-minimal Lagrangian torus in $\mathbb{CP}^3$.
\end{thm}

\section*{Acknowledgment}
A.E.Mironov is grateful to generous supports by KIAS, where part of
the work was done in KIAS. The work of A.E.Mironov was partially
supported by Russian Federation foundation for basic research (grant
no.06-01-00094a) and grant MK-9651.2006.1 of the president of
Russian Federation.

Zuo would like to thank Qing Chen and  Bumsig Kim and Youjin Zhang for their
constant guidance and supports. Zuo  also thanks Chengmin Bai for hospitality
during stay at Chern Institute of mathematics in Nankai university, where part
of the work was done. The work of Zuo was partially supported by a post-doc
fellowship from KIAS and the Natural Science Foundation of China
(grant no. 10501043).

Both of the authors thank organizers for the hospitality during the ISLAND-3 on
Islay, Scotland.

%%%%%%%%%%%%%%%%%%%%%%%%%%%%%%%%%%%%%%%%%%%%%%%%%%%%%%%%%%%%%%%%%

\end{document}